\documentclass[11pt]{amsart}

\usepackage[pdftex]{graphicx}
\usepackage[a4paper,margin=2.5cm]{geometry}

\usepackage{amsfonts}
\usepackage{amsmath}
\usepackage{amssymb}
\usepackage{amsthm}
\usepackage{dsfont}
\usepackage{tikz}
\usepackage{mathdots}

\usepackage{rotating}

\usepackage[T1]{fontenc}

\usepackage{paralist}
\usepackage{color}

\usepackage{overpic}

\usepackage[colorlinks,cite color=blue,pagebackref=true,pdftex]{hyperref}

\renewcommand\emptyset{\varnothing}

\newcommand\R{\mathbb{R}}

\newcommand\V{\mathsf{V}}

\newcommand\cone{\mathrm{cone}}

\renewcommand\int{\mathsf{int}}

\newtheorem{thm}{Theorem}[section]
\newtheorem{cor}[thm]{Corollary}
\newtheorem{lem}[thm]{Lemma}
\newtheorem{prop}[thm]{Proposition}
\newtheorem{conj}[thm]{Conjecture}
\newtheorem{quest}{Question}

\theoremstyle{definition}

\newtheorem{example}[thm]{Example}
\newtheorem{rem}[thm]{Remark}

\title{Linear recursions for integer point transforms}

\author{Katharina Jochemko}
\address{Department of Mathematics,
Royal Institute of Technology, %
SE--100 44 Stockholm, %
Sweden}
\email{jochemko@kth.se}


\date{\today}

\begin{document}

\begin{abstract}
We consider the integer point transform $\sigma _P (\mathbf{x}) = \sum _{\mathbf{m} \in P\cap \mathbb{Z}^n} \mathbf{x}^\mathbf{m} \in \mathbb C [x_1^{\pm 1},\ldots, x_n^{\pm 1}]$ of a polytope $P\subset \mathbb{R}^n$. We show that if $P$ is a lattice polytope then for any polytope $Q$ the sequence $\lbrace \sigma _{kP+Q}(\mathbf{x})\rbrace _{k\geq 0}$ satisfies a multivariate linear recursion that only depends on the vertices of $P$. We recover Brion's Theorem and by applying our results to Schur polynomials we disprove a conjecture of Alexandersson (2014).
\end{abstract}

\maketitle

\section{Introduction}\label{sec:intro}

A \textbf{polytope} is the convex hull of finitely many points in $\mathbb{R}^n$.
A polytope is a \textbf{lattice polytope} if all its vertices lie in the integer lattice $\mathbb{Z}^n$. The \textbf{integer point transform} of a polytope $P$ is defined by
\[
\sigma _P (x)=\sum _{\mathbf{m} \in P\cap \mathbb{Z}^n}\mathbf{x}^\mathbf{m} \in \mathbb{C}[x_1 ^{\pm 1},\ldots, x_n ^{\pm 1}] \, ,
\]
where $\mathbf{x}^\mathbf{m}$ denotes $x_1 ^{m_1}\cdots x_n ^{m_n}$ for all $\mathbf{m}\in \mathbb{Z}^n$. In this note we study sequences $\{\sigma _{kP} (\mathbf{x})\}_{k\geq 0}$ of integer point transforms of integer dilates of polytopes $P$ and relatives. We prove the following linear recursion.

\begin{thm}\label{thm:reclatsum}
Let $Q$ be a polytope in $\mathbb{R}^n$ and let $P$ be a lattice polytope with vertex set $\V(P)=\{v_1,\ldots,v_r\}$. Then the sequence $\lbrace \sigma _{kP+Q} (\mathbf{x})\rbrace _{k\geq 0}$  satisfies the linear recursion 
\begin{eqnarray*}
\sigma_{(k+r)P+Q} (\mathbf{x})&=&\sum _{\emptyset \neq I\subseteq [r]} (-1)^{1+\left|I\right|} \mathbf{x}^{\sum _{i\in I} v_i}\sigma_{(k+r-\left|I\right|)P +Q}(\mathbf{x}) \, 
\end{eqnarray*}
with characteristic polynomial
\[
 \chi _{P;Q} (X) \ := \ \prod _{\mathbf{v}\in \V(P)} (X-\mathbf{x}^{\mathbf{v}}) \, .
\]
If $Q$ is a lattice polytope, then  $\chi _{P;Q}$ is minimal.
\end{thm}

In particular, the recursion only depends on the vertices of $P$. This improves results by Alexandersson~\cite{AlexanderssonIDP} where it was assumed that $P$ has the integer decomposition property and $Q=\{0\}$.

Employing classical results from valuation theory, in Section~\ref{sec:charfunc} we first prove a recursion for indicator functions of dilated polytopes. Then, in Section~\ref{sec:multirec}, we apply these results to integer point transforms and prove Theorem~\ref{thm:reclatsum}. We recover Brion's Theorem in Section~\ref{sec:Brion} and by applying our results to Schur polynomials we disprove a conjecture of Alexandersson~\cite{Alexandersson} in Section~\ref{sec:Schur}.

\section{Characteristic functions and valuations}\label{sec:charfunc}
In this section we prove a linear recursion for indicator functions of integer dilates of a polytope $P$. Let $\mathcal{P}$ denote the set of polytopes in $\mathbb{R}^n$ and let $G$ be an abelian group. A \textbf{valuation} is a map $\varphi \colon \mathcal{P} \rightarrow G$ such that $\varphi (\emptyset )=0$ and
\[
\varphi (P\cup Q)=\varphi (P) +\varphi(Q)-\varphi (P\cap Q) \, ,
\]
for all $P,Q\in \mathcal{P}$ such that also $P\cup Q\in \mathcal{P}$. The volume, the number of lattice points inside a polytope and the integer point transform are examples of valuations. It was shown by Volland~\cite{Volland} that every valuation satisfies the inclusion-exclusion property. That is, for polytopes $P,P_1,\ldots,P_r$ such that $P=P_1\cup \cdots \cup P_r$
\[
\varphi (P) = \sum _{\emptyset\neq I\subseteq [r]} (-1)^{|I|+1}\varphi(P_I)\, ,
\]
where $P_I:=\bigcap _{i\in I}P_i$. Stronger even, it follows from a result of Groemer~\cite{Groemer}, that if $\sum \alpha _i \mathbf{1}_{P_i}=0$ for polytopes $P_1,\ldots,P_m$ and some $\alpha _1,\ldots,\alpha _m\in \mathbb{Z}$ then $\sum _i \alpha_i\varphi (P_i)=0$ where $\mathbf{1}_{P}$ denotes the indicator function for every polytope $P$. A function of the form $\sum \alpha _i \mathbf{1}_{P_i}$ is called a \textbf{polytopal simple function}. By Groemer's result, every valuation uniquely defines a homomorphism from the abelian group of polytopal simple functions to $G$, that is, every polytope can be identified with its indicator function. For valuations on lattice polytopes this was proved by McMullen~\cite{mcmullen2009valuations}. It is well-known that for every affine linear map $T\colon \mathbb{R}^n \rightarrow \mathbb{R}^m$
\begin{equation}\label{eq:pushforward}
\mathbf{1}_P\mapsto \mathbf{1}_{T(P)}
\end{equation}
defines a valuation (see, e.g., \cite[Chapter 8]{BarvinokConv}). Using this push forward map we obtain the following recursion on indicator functions.
\begin{thm}\label{thm:char}
Let $P$ be a polytope in $\mathbb{R}^n$ and $\lbrace v_1,\ldots ,v_r\rbrace = V(P)$ the vertex set of $P$. Then 
\begin{equation} \label{eqn:charrec}
\mathbf{1}_{(k+r)P}=\sum _{\emptyset \neq I\subseteq [r]} (-1)^{1+\left|I\right|} \mathbf{1}_{Q^{k,r}_I}
\end{equation}
for all $k\geq 0$ where $Q^{k,r}_I=(k+r-\left|I\right|)P+\sum _{i\in I} v_i$.
\end{thm}
\begin{proof}
We first assume that $P$ is the $(d-1)$-dimensional standard simplex $\Delta _{d-1}=\lbrace x\in \mathbb{R}^d \colon x_1+\cdots +x_d = 1,  x_1,\ldots ,x_d \geq 0\rbrace$. Its $(k+d)$-th dilate is given by
\[
(k+d)\Delta _{d-1}=\lbrace x\in \mathbb{R}^d \colon x_1+\cdots +x_d = d+k,  x_1,\ldots ,x_d \geq 0\rbrace
\]
For all $I\subseteq [d]$, let
\[
P_I = (k+d)\Delta _{d-1}\cap \lbrace x\in \mathbb{R}^d \colon x_i \geq 1 \text{ for all } i\in I\rbrace \, .
\]
Then $P_I=\bigcap _{i\in I} P_{\lbrace i\rbrace }$ for all $\emptyset \neq I\subseteq [r]$. As in~\cite{Sam} we observe that $(k+d)\Delta _{d-1}=P_{\emptyset}=\bigcup _{i\in [d]} P_{\lbrace i\rbrace }$ for all $k\geq 0$. Therefore, by inclusion-exclusion,
\[
\mathbf{1}_{(k+d)\Delta _{d-1}}=\sum _{\emptyset \neq I\subseteq [d]} (-1)^{1+\left|I\right|} \mathbf{1}_{P_I}\, 
\]
and we finish the proof of this case by observing that $P_I=(k+d-\left|I\right|)\Delta _{d-1} + \sum _{i\in I} e_i$. 

For the general case, we recall that every polytope is an affine linear projection of a standard simplex and, thus, the claim follows by applying the push forward map~\eqref{eq:pushforward}.\qedhere

\end{proof}

For fixed $Q\in \mathcal{P}$, $\mathbf{1}_P\mapsto \mathbf{1}_{P+Q}$ defines a valuation (see, e.g.,~\cite{Schneider}) where $P+Q=\{p+q\colon p\in P, q\in Q\}$ is the \textbf{Minkowski sum}. The family of all polytopal simple functions forms an algebra where the multiplicative structure is given by the Minkowski sum of polytopes: $\mathbf{1}_{P}\star \mathbf{1}_{Q}:=\mathbf{1}_{P+Q}$ for all polytopes $P$ and $Q$. Another proof of Theorem~\ref{thm:char} can be obtained from the following result which was proved in~\cite{Lawrence06} . See also~\cite{pukhlikov1992finitely} for related material.

\begin{thm}[{\cite[Lemma 5]{Lawrence06}}]\label{thm:kp}
Let $P$ be a polytope and $\lbrace v_1,\ldots ,v_r\rbrace= V(P)$ the set of vertices of $P$. Then
\begin{equation}\label{eq:kp}
\left(\mathbf{1}_P -\mathbf{1}_{v_1}\right) \star \cdots \star \left(\mathbf{1}_P -\mathbf{1}_{v_r}\right) =0 \, .
\end{equation}
\end{thm}
\begin{proof}[$2$nd proof of Theorem~\ref{thm:char}]
The proof follows from Theorem~\ref{thm:kp} by expanding equation~\eqref{eq:kp} and multiplying both sides with $\mathbf{1}_{kP}$.
\end{proof}

By the discussion above, Theorem~\ref{thm:char} is equivalent to the following.
\begin{thm}\label{thm:valuations}
Let $P$ be a polytope in $\mathbb{R}^n$ and $\lbrace v_1,\ldots ,v_r\rbrace = V(P)$ the vertex set of $P$, and $\varphi\colon\mathcal{P}\rightarrow G$ a valuation. Then
\begin{equation} \label{eqn:charrec}
\varphi((k+r)P)=\sum _{\emptyset \neq I\subseteq [r]} (-1)^{1+\left|I\right|} \varphi(Q^{k,r}_I)
\end{equation}
for all $k\geq 0$ where $Q^{k,r}_I=(k+r-\left|I\right|)P+\sum _{i\in I} v_i$.
\end{thm}

\section{A multivariate recursion}\label{sec:multirec}
A sequence $\mathbf{a}=\lbrace a_k\rbrace _{k\geq 0}$ of elements in $\mathbb{C}(x_1,\ldots, x_n)$ satisfies a \textbf{linear recursion} of order $d\geq 1$ if there are $c_1,\ldots, c_d\in \mathbb{C}(x_1,\ldots, x_n )$, $c_d\neq 0$, such that
\[
a_{k}=\sum_{j=1}^{d} c_j a_{k-j}
\]
for all $k\geq d$. The corresponding \textbf{characteristic polynomial} $\chi_{\mathbf{c}}$ is defined as $X^d -\sum _{j=1}^d c_j X^{d-j} \in \mathbb{C}(x_1,\ldots, x_n)[X]$. The polynomial $\chi _\mathbf{c}$ is called \textbf{minimal} if for every vector $\mathbf{c}'=(c_1',\ldots, c_{d'}')$  corresponding to a linear recursion of $\mathbf{a}$ we have $\chi _{\mathbf{c}}|\chi _{\mathbf{c}'}$. Since $\mathbb{C}(x_1,\ldots, x_n)[X]$ is a principal ideal domain a uniquely determined minimal polynomial exists.

We are now ready to proof Theorem~\ref{thm:reclatsum}.
\begin{proof}[Proof of Theorem~\ref{thm:reclatsum}]
Let $r=\left|V(P)\right|$ be the number of vertices of $P$.
Since the maps $P\mapsto P+Q$  and also $P\mapsto \sigma _P (x)$ define valuations, by Theorem~\ref{thm:valuations}
\begin{eqnarray*}
\sigma_{(k+r)P+Q} (\mathbf{x})&=&\sum _{\emptyset \neq I\subseteq [r]} (-1)^{1+\left|I\right|} \sigma_{(k+r-\left|I\right|)P+\sum _{i\in I}v_i +Q}(\mathbf{x})\\
&=&\sum _{\emptyset \neq I\subseteq [r]} (-1)^{1+\left|I\right|} \mathbf{x}^{\sum _{i\in I} v_i}\sigma_{(k+r-\left|I\right|)P +Q}(\mathbf{x}) \, ,
\end{eqnarray*}
where the last equation follows by observing that $\sigma _{P+v}(\mathbf{x})=\mathbf{x}^v \sigma _{P}(\mathbf{x})$ for all $v\in \mathbb{Z}^n$. We observe that $\chi_{P;Q}$ is the characteristic polynomial of this linear recursion.
  
Now let $Q$ be a lattice polytope and suppose that $\chi _{P;Q}$ is not minimal. Then, for some vertex $\mathbf{u}$ of $P$, $\lbrace \sigma _{kP+Q} (\mathbf{x})\rbrace _{k\geq 0}$ satisfies a linear recursion with characteristic polynomial $\prod _{\mathbf{v}\in \V(P)\setminus \lbrace \mathbf{u}\rbrace } (X-\mathbf{x}^{\mathbf{v}})$. That is
\[
\sigma_{(k+r)P+Q}(\mathbf{x})+\sum_{j=1}^{\left|\V(P)\right|-1} (-1)^j e_j(\lbrace \mathbf{x}^\mathbf{v} \colon \mathbf{v}\in \V(P)\setminus \{\mathbf{u}\}\rbrace) \sigma_{(k+r-j)P+Q}(\mathbf{x}) =0
\]
where $e_j$ denotes the $j$-th elementary symmetric polynomial in $\left|\V(P)\right|-1$ variables. Now let $\mathbf{v}$ be a vertex  of $Q$ such that $\mathbf{u}+\mathbf{v}$ is a vertex of $P+Q$. Then $(k+r)\mathbf{u}+\mathbf{v}$ is a vertex of $(k+r)P+Q$ and thus $x^{(k+r)\mathbf{u}+\mathbf{v}}$ appears as a summand in $\sigma_{(k+r)P+Q}(x)$. However, it does not appear in $e_j(\lbrace \mathbf{x}^{\mathbf{v}} \colon \mathbf{v}\in V(P)\setminus \{\mathbf{u}\}\rbrace) \sigma_{(k+r-j)P+Q}(x)$ for any $1\leq j\leq\left|\V(P)\right|-1$. To see that, it suffices to argue that $(k+r)\mathbf{u}+\mathbf{v}$ is not contained in  $(k+r-j)P+Q+\sum _{l=1}^jv_l$ for any choice of $v_1,\ldots,v_j\in V(P)\setminus \{\mathbf{u}\}$. For that, let $\ell \colon \mathbb{R}^n\rightarrow \mathbb{R}$ be a linear functional such that $\ell (\mathbf {u})>\ell (p)$ for all $p\neq \mathbf{u}$ in $P$ and $\ell (\mathbf {v})>\ell (q)$ for all $q\neq \mathbf{v}$ in $Q$. Then $\ell((k+r-j)p+q+\sum _{l=1}^jv_l)=(k+r-j)\ell (p)+\ell(q)+\sum _{l=1}^j\ell(v_l)<(k+r-j)\ell(\mathbf{u})+\ell(\mathbf{v})+j\ell(\mathbf{u})=\ell((k+r)\mathbf{u}+\mathbf{v})$ for all $p\in P$ and $q\in Q$ and the conclusion follows.
\end{proof}

Every linear map $f\colon \mathbb{R}^n\rightarrow \mathbb{R}^l$ with the property that $f(\mathbb{Z}^n)\subseteq \mathbb{Z}^l$ induces an algebra homomorphism
\begin{eqnarray*}
\bar f\colon \mathbb{C}\left[x_1^{\pm 1},\ldots ,x_n^{\pm 1}\right]  &  \rightarrow  &  \mathbb{C}\left[x_1^{\pm 1},\ldots, x_l^{\pm 1}\right]\\
\mathbf{x}^\mathbf{m}  &  \mapsto  &  \mathbf{x}^{f(\mathbf{m})}
\end{eqnarray*}
As a consequence of Theorem~\ref{thm:reclatsum} we therefore obtain the following.
\begin{prop}\label{prop:projection}
Let $Q$ be a polytope in $\mathbb{R}^n$ and $P$ be a lattice polytope, and let $f\colon \mathbb{R}^n \rightarrow \mathbb{R}^l$ a linear map such that  $f(\mathbb{Z}^n)\subseteq \mathbb{Z}^l$. Then $\lbrace \bar f(\sigma _{kP+Q}(\mathbf{x}))\rbrace _{k\geq 0}$ satisfies a linear recursion with characteristic polynomial 
\[
\chi _{P;Q}^f(X):=\prod _{\mathbf{v}\in V(P)} (X-x^{f(\mathbf{v})}).
\]
\end{prop}

The following two examples show that the minimality of a characteristic polynomial is not necessarily preserved under affine transformations or taking Minkowski sums.

\begin{example}
If $Q$ in Theorem~\ref{thm:reclatsum} is not a lattice polytope then $\chi _{P;Q}$ is not necessarily minimal. A counterexample is given by the lattice segment $P=[0,1]$ and the point $Q=\lbrace (0.5,0.5)\rbrace$ in $\R^2$. In that case $\sigma_{kP+Q}\equiv 0$ is constant.
\end{example}

\begin{example}[Ehrhart polynomials]
For $f\colon \R^n\rightarrow \R^0$ and $f\equiv 0$ we obtain $\bar f(\sigma _{kP}(\mathbf{x}))=|kP\cap \mathbb{Z}^n|$ and thus recover the Ehrhart function counting lattice points in integer dilates of $P$. If $P$ is a lattice polytope then this function is known to agree with a polynomial of degree $\dim P$~\cite{ehrhartpolynomial}. Therefore the order of the minimal polynomial of the sequence is $\dim P$ as was demonstrated in~\cite{Sam} and is thus in general smaller than $\left|\V(P)\right|$.
\end{example}
These examples motivate the following question.
\begin{quest}
What are necessary and sufficient conditions on $Q$ and on $f$ that guarantee that $\chi _{P;Q}^f$ is minimal?
\end{quest}

\section{Brion's Theorem}\label{sec:Brion}
In this section we provide a proof of Brion's Theorem using the recursion given in Theorem~\ref{thm:reclatsum}. For a polytope $P\subseteq \R^n$ and a vertex $\mathbf{v}$ of $P$ the \textbf{tangent cone} $\mathcal{K}_{\mathbf{v}}$ is defined as $\left\{\mathbf{v}+\mathbf{w}\colon \mathbf{v}+\varepsilon \mathbf{w}\in P \text{ for } 0<\varepsilon \ll 1\right\}$. If the polytope $P$ has rational edge directions, in particular, if it is a lattice polytope, then the  integer point transform of $\mathcal{K}_{\mathbf{v}}$ is a rational function. 
\begin{thm}[Brion's Theorem \cite{brion1988points}]\label{thm:brion}
Let $P$ be a lattice polytope. Then
\[
\sigma _P (\mathbf{x}) =\sum _{\mathbf{v} \in \V(P)}\sigma _{\mathcal{K}_\mathbf{v}} (\mathbf{x})
\]
as rational functions.
\end{thm}

The following is an immediate consequence of~\cite[Lemma 13.5.]{barvinokinteger}

\begin{lem}{\cite{barvinokinteger}}\label{lem:analytic}
Let $u_1,\ldots, u_k\in \mathbb{Z}^n$ such that the cone $\mathcal{K}:=\cone (u_1,\ldots, u_k)$ generated by $u_1,\ldots, u_k$ is pointed. Then
\[
\sigma _{\mathcal{K}} (\mathbf{x})=\sum _{\mathbf{m}\in \mathcal{K} \cap \mathbb{Z}^n} \mathbf{x}^\mathbf{m}
\]
is a rational function and converges absolutely for all $\mathbf{x}$ in $\lbrace \mathbf{x}\in \mathbb{C}^n \colon |\mathbf{x}^{u_i}|<1 \text{ for } i=1,\ldots , k \rbrace$.
\end{lem}

A further ingredient for our proof of Brion's Theorem is the following well-known result (see, e.g., \cite[Chapter 5]{spiegel1971schaum}).
\begin{lem}\label{lem:expl}
Let $\lbrace a_n\rbrace _{n\in \mathbb{N}}$ be a sequence of elements of a field $K$ that satisfy a linear recursion of order $d$ with characteristic polynomial
\[
\prod _{i=1} ^d (X-r_i).
\]
If all roots $r_1,\ldots, r_d$ are distinct then there are $\alpha _1, \ldots, \alpha _d \in K$ such that
\[
a_n =\sum _{i=1} ^d \alpha _i r_i ^n
\]
for all $n \in \mathbb{N}$.
\end{lem}
\begin{proof}[Proof of Theorem \ref{thm:brion}]
By Theorem \ref{thm:reclatsum} and Lemma \ref{lem:expl} there are $c_\mathbf{v} \in \mathbb{C}(x_1,\ldots , x_n)$ for all $\mathbf{v}\in \V(P)$ such 
\begin{equation}\label{eqn:converg}
\sigma _{kP}(\mathbf{x})=\sum _{\mathbf{v}\in \V(P)} c_\mathbf{v} \mathbf{x}^{k\mathbf{v}}
\end{equation}
for all $k\geq 0$. Our goal is to show that $c_\mathbf{w} \cdot \mathbf{x}^{\mathbf{w}}=\sigma _{\mathcal{K}_\mathbf{w}}(\mathbf{x})$ as rational functions for all $\mathbf{w}\in \V(P)$, or, equivalently, that $c_\mathbf{w}$ equals the integer point transform of the tangent cone $\tilde{\mathcal{K}}_0$ of the vertex $0$ of the translated polytope $P-\mathbf{w}$. Equation \eqref{eqn:converg} is equivalent to $\sigma _{k(P-\mathbf{w})}(\mathbf{x})=\sum _{\mathbf{v}\in \V(P)} c_\mathbf{v} \mathbf{x}^{k(\mathbf{v}-\mathbf{w})}$. As $k$ goes to infinity $\sigma _{k(P-\mathbf{w})}(\mathbf{x})$ converges absolutely to $\sigma _{\tilde{\mathcal{K}}_0}(\mathbf{x})$ on $W_{\tilde{\mathcal{K}}_0}=\lbrace \mathbf{x}\in \mathbb{C}^n \colon |\mathbf{x}^{\mathbf{v}-\mathbf{w}}|<1 \text{ for all } \mathbf{v}\in V(P) \setminus \lbrace \mathbf{w}\rbrace\rbrace$  by Lemma \ref{lem:analytic}.  On the other hand, $\sum _{\mathbf{v}\in \V(P)} c_\mathbf{v} \mathbf{x}^{k(\mathbf{v}-\mathbf{w})}$ converges to $c_\mathbf{w}$. Thus $\sigma _{\tilde{\mathcal{K}}_0}(x)$ and $c_\mathbf{w}$ coincide on $W_{\tilde{\mathcal{K}}_0}$ and  are therefore the same as rational functions.
\end{proof}

\section{Schur polynomials}\label{sec:Schur}
In this section we apply our results to Schur polynomials.

A \textbf{partition} is a vector $\pmb{\lambda }=(\lambda _1 \geq \lambda _2\geq \ldots \geq \lambda _n)$ of weakly decreasing nonnegative integers. The number of strictly positive entries $\lambda _i$ is called the \textbf{length} of $\pmb{\lambda}$. A partition $\pmb{\mu}$ is smaller than a partition $\pmb{\lambda}$ with respect to the \textbf{inclusion order} if $\mu _i \leq \lambda _i$ for all $i$. The partition $\pmb{\mu}$ is smaller than a partition $\pmb{\lambda}$ with respect to the \textbf{domination order}, denoted $\pmb{\lambda}\unrhd\pmb{\mu}$, if $\sum _{i= 1}^n \lambda _i = \sum _{i= 1}^n \mu _i$ and $\sum _{i= 1}^k \lambda _i \geq \sum _{i=1}^k \mu _i$ for all $k$. A \textbf{skew Young diagram} of shape $\pmb{\lambda} / \pmb{\mu}$ is an axes-parallel arrangement of unit squares in the plane centered at the coordinates $\left\{(i,j)\in \mathbb{Z}^2\colon \mu_i <j\leq \lambda _i\right\}$. A \textbf{semi-standard Young tableau} is a Young diagram together with a filling of the boxes with natural numbers such that the numbers are strictly increasing in each column and weakly increasing in every row. Let $\mathbb{T}^n_{\pmb\lambda / \pmb\mu}$ denote the set of semi-standard Young tableaux filled with numbers in $[n]=\{1,2,\ldots,n\}$. For every $T$ in $\mathbb{T}^n_{\pmb\lambda / \pmb\mu}$ let $w(T)$ be the vector $\mathbf{t}=(t_1,\ldots, t_n)$ where $t_i$ is the number of boxes filled with $i$. The vector $w(T)$ is called the \textbf{weight} of $T$. The \textbf{Kostka coefficient} $K_{\pmb\lambda / \pmb\mu, \mathbf{w}}$ equals the number of tableaux of shape $\pmb\lambda / \pmb\mu$ with weight $\mathbf{w}$. In particular, $K_{\pmb\lambda / \pmb\mu, \mathbf{w}}>0$ if and only if there is $T\in \mathbb{T}^n_{\pmb\lambda / \pmb\mu}$ with $w(T)=\mathbf{w}$.  The \textbf{skew Schur polynomial} of shape $\pmb{\lambda}/\pmb{\mu}$ is defined as
\[
s_{\pmb\lambda /\pmb\mu }(\mathbf{x}) \ = \ \sum _{T\in \mathbb{T}^n_{\pmb\lambda / \pmb\mu}}\mathbf{x}^{w(T)} \in \mathbb C [x_1^{\pm 1},\ldots, x_n^{\pm 1}]\, .
\]
In~\cite{Alexandersson} Alexandersson proved the following recursion for Schur polynomials. 
\begin{thm}[{\cite[Theorem 1]{Alexandersson}}]\label{thm:SchurRec}
Let $n$ be a natural number and let $\pmb{\kappa},\pmb{\lambda},\pmb{\mu},\pmb{\nu}$ be partitions of length at most $n$ such that $\pmb\lambda \supseteq \pmb\mu$ and $\pmb\kappa +k \pmb\lambda \supseteq \pmb\nu + k \pmb\mu$ for some positive integer $k$. Then there is a natural number $r$ such that the sequence $\{s_{\pmb\kappa +l\pmb\lambda /\pmb\nu +l\pmb\mu}(\mathbf{x})\}_{l=r}^\infty$ satisfies a linear recursion with characteristic polynomial
\[
\chi (X) \ = \ \prod _{T\in \mathbb{T}^n_{\pmb\lambda / \pmb\mu}}\left(X-\mathbf{x}^{w(T)}\right) \, .
\]
\end{thm}

Furthermore, in~\cite{Alexandersson} the following conjecture concerning the minimal polynomial was stated. For every vector $\mathbf{w}$ let $\overline{\mathbf{w}}$ denote the vector obtained from $\mathbf{w}$ by rearranging its coordinates in non-increasing order.

\begin{conj}[{\cite[Conjecture 25]{Alexandersson}}]\label{conj:SchurMinimal}
Let $\pmb{\kappa},\pmb{\lambda},\pmb{\mu},\pmb{\nu}$ be as in Theorem~\ref{thm:SchurRec} and let
\[
W \ = \ \{\mathbf{w}\in \mathbb{N}^n\colon K_{\pmb{\lambda}/\pmb{\mu},\mathbf{w}}>0 \text{ and }\overline{\mathbf{w}}\unrhd \overline{\pmb{\lambda}-\pmb{\mu}}\} \, .
\]
Then, for sufficiently large $r$, $\{s_{\kappa +l\mu /\lambda +l\nu}(\mathbf{x})\}_{l=r}^\infty$ satisfies a linear recursion with minimal polynomial
\[
\chi (X) \ = \ \prod _{\mathbf{w}\in W}\left(X-\mathbf{x}^{\mathbf{w}}\right) \, .
\]
\end{conj} 

We use Theorem~\ref{thm:reclatsum} and a well-known correspondence between elements in $\mathbb{T}^n_{\pmb\lambda / \pmb\mu}$ and lattice points in the Gelfand-Tsetlin-Polytope $\mathbf{GL}_{\pmb\lambda / \pmb\mu}$ to improve Theorem~\ref{thm:SchurRec} and to given an example in which the polynomial in Conjecture~\ref{conj:SchurMinimal} is not minimal thus refuting the conjecture. 

\begin{figure}
\begin{center}
\begin{minipage}{0.4\textwidth}
\begin{equation*}
\begin{array}{ccccccccc}
   &       & & x_{1,1} & & x_{1,2} &\ldots & x_{1,n}  &   \\
          &&x_{2,1}&   &x_{2,2}&\ldots   &x_{2,n}&    &  \\
                    &\iddots&   &\iddots&   &\iddots&    &&  \\
        x_{n+1,1} & & x_{n+1,2} &\ldots & x_{n+1,n} & &  & & 
    \end{array}
\end{equation*}
\end{minipage}
\hspace*{1cm}
\begin{minipage}{0.4\textwidth}
\begin{equation*}
\begin{array}{ccccccccc}
   &       & & 1 & & 3 & & 5  &   \\
          &&1&   &3&   &5&    &  \\
                    &1&   &2&   &4&    &&  \\
        1 & & 1 & & 4 & &  & & 
    \end{array}
\end{equation*}
\end{minipage}
\end{center}
\caption{Gelfand-Tsetlin patterns.}
\label{fig:GTPgeneral}
\end{figure}

There is a one-to-one correspondence between semi-standard Young tableux and Gelfand-Tsetlin patterns. A \textbf{Gelfand-Tsetlin pattern} is a rectangular array of nonnegative real numbers $\{x_{i,j}\}_{i=1,\ldots,n+1\atop j=1,\ldots,n}$ arranged as in Figure~\ref{fig:GTPgeneral} such that the entries are weakly increasing in north-east and south-east direction, that is $x_{i,j}\leq x_{i+1,j+1}$ for all $i,j$ and $x_{i,j}\leq x_{l,j}$ for all $i>l$. For fixed top and bottom rows the family of Gelfand-Tsetlin patterns forms a polytope, the \textbf{Gelfand-Tsetlin polytope}, which belongs to the class of marked order polytopes introduced by Ardila, Bliem and Salazar~\cite{ArdilaBliemSalazar}. There is a well-known one-to-one correspondence between elements of $\mathbb{T}^n_{\pmb\lambda / \pmb\mu}$ and integer valued Gelfand-Tsetlin patterns with top row $\pmb \lambda$ and bottom row $\pmb \mu$, that is, lattice points in the corresponding Gelfand-Tsetlin polytope $\mathbf{GL}_{\pmb\lambda / \pmb\mu}$.  Via this correspondence, the weight function can be represented as a linear function on $\mathbf{GL}_{\pmb\lambda / \pmb\mu}$, namely for every Gelfand-Tsetlin pattern $\mathbf{x}=\{x_{i,j}\}$ the $i$-th coordinate of the weight $w(\mathbf{x})$ equals $\sum _{k=1}^n (x_{i,k}-x_{i+1,k})$ for all $1\leq i\leq n$. Further details may be found in ~\cite{EC2}. It follows that
\[
s_{\pmb\lambda /\pmb\mu }(\mathbf{x}) \ = \ \sum _{p}\mathbf{x}^{w(p)} \, ,
\]
where $p$ is over all lattice points in $\mathbf{GL}_{\pmb\lambda / \pmb\mu}$.

As a corollary of Theorem~\ref{thm:reclatsum} we obtain the following.
\begin{cor}\label{cor:Minimal}
Let $n$ be a natural number and let $\pmb{\kappa},\pmb{\lambda},\pmb{\mu},\pmb{\nu}$ be partitions of length at most $n$ such that $\pmb\lambda \supseteq \pmb\mu$ and $\pmb\kappa +k \pmb\lambda \supseteq \pmb\nu + k \pmb\mu$ for some positive integer $k$. Let $V$ be the set of vertices of $\mathbf{GL}_{\pmb\lambda / \pmb\mu}$. Then there is an integer $r\gg 0$ such that $\{s_{\pmb\kappa +l\pmb\lambda /\pmb\nu +l\pmb\mu}(\mathbf{x})\}_{l=r}^\infty$ satisfies a linear recursion with characteristic polynomial
\[
\chi (X) \ = \ \prod _{\mathbf{v}\in V}\left(X-\mathbf{x}^{w(\mathbf{v})}\right) \, .
\]
\end{cor}
\begin{proof}
Let $\mathbf{f}=(\pmb{\lambda},\pmb{\mu})$ and $\mathbf{g}=(\pmb{\kappa},\pmb{\nu})$. Then there is an $r\gg 0$ such that if $f_i<f_j$ then $rf_i+g_i<rf_j+g_j$ for all $i\neq j$. In particular, one can find a permutation $\sigma \in S_{2n}$ such that
\[
f_{\sigma(1)}\leq f_{\sigma(2)}\leq \cdots \leq f_{\sigma(2n)} \quad \text{and} \quad rf_{\sigma(1)}+g_{\sigma (1)}\leq \cdots \leq rf_{\sigma(2n)}+g_{\sigma(2n)} \, .
\]
Then, by Theorem~\cite[Theorem 2.10]{fang2018minkowski},
\[
\mathbf{GL}_{\pmb\kappa +l\pmb \lambda / \pmb \nu +l\pmb \mu } \ = \ \mathbf{GL}_{\pmb\kappa +r\pmb \lambda / \pmb \nu +r\pmb \mu } + (l-r)\mathbf{GL}_{\pmb\lambda / \pmb\mu} \, 
\]
for all $l\geq r$. The claim now follows from Proposition~\ref{prop:projection} since the weight function $w$ is linear.
\end{proof}
Since typically there are more lattice points in $\mathbf{GL}_{\pmb\lambda / \pmb\mu}$ than vertices, Corollary~\ref{cor:Minimal} shows that the characteristic polynomial given in Theorem~\ref{thm:SchurRec} is in general not minimal. The next example shows that also the polynomial given in Conjecture~\ref{conj:SchurMinimal} is not minimal in general, thus refuting it.

\begin{example}\label{example}
Let $n=3$, $\pmb \lambda = (5,3,1)$ and $\pmb \mu = (3,0,0)$. Consider the skew Young tableau $T$ and its corresponding Gelfand-Tsetlin pattern $p$ depicted in Figure~\ref{fig:GTP}. Then 
\[
w(T)=w(p)=(4,2,0)\unrhd (3,2,1)=\overline{\pmb \lambda -\pmb \mu}\, .
\]
From the face structure studied in~\cite{JochemkoSanyal14,pegel2018face} it follows that the coordinates of any vertex of $\mathbf{GL}_{\pmb\lambda / \pmb\mu}$ are in the set $\{0,1,3,5\}$. Let $\mathbf{x}=\{x_{i,j}\}$ be a Gelfand-Tsetlin pattern that is a vertex of $\mathbf{GL}_{\pmb\lambda / \pmb\mu}$. Then $x_{4,1},x_{4,2}$ and $x_{3,1}$ are $0$. Furthermore, $x_{2,1}\in \{0,1\}$. If $x_{2,1}=0$, then the sum of entries of the first row of $\mathbf{x}$ is odd and the sum of entries of the second is even, therefore $w(\mathbf{x})_1$ is odd and $w(\mathbf{x})\neq (4,2,0)$. On the other hand, if $x_{2,1}=1$, then $x_{3,2}\in \{1,3\}$ and in that case $w(\mathbf{x})_2$ is odd and again $w(\mathbf{x})\neq (4,2,0)$. In summary, $(4,2,0)\in W$ is not the weight of a vertex of $\mathbf{GL}_{\pmb\lambda / \pmb\mu}$ and therefore 
\[
\prod _{\mathbf{w}\in W}\left(X-\mathbf{x}^{\mathbf{w}}\right) \nmid \prod _{\mathbf{v}\in V}\left(X-\mathbf{x}^{w(\mathbf{v})}\right) \, .
\]
Therefore, by Corollary~\ref{cor:Minimal}, $\prod _{\mathbf{w}\in W}\left(X-\mathbf{x}^{\mathbf{w}}\right)$ cannot be the minimal polynomial.
\begin{figure}[h]
\begin{center}
\begin{minipage}{0.4\textwidth}
\begin{tikzpicture}[scale=0.8]
\draw (0,0) -- (1,0) -- (1,1) -- (3,1) -- (3,3) -- (5,3) -- (5,2) -- (0,2) -- (0,0);
\draw (0,1) -- (1,1);
\draw (1,1) -- (1,2);
\draw (2,1) -- (2,2);
\draw (4,2) -- (4,3);
\draw (0.5,0.5) node {$3$};
\draw (1.5,1.5) node {$3$};
\draw (2.5,1.5) node {$3$};
\draw (4.5,2.5) node {$3$};
\draw (0.5,1.5) node {$2$};
\draw (3.5,2.5) node {$2$};
\draw (2.5,-1) node {$T$};
\end{tikzpicture}
\end{minipage}
\begin{minipage}{0.4\textwidth}
\begin{equation*}
\begin{array}{ccccccccc}
          & & &\mathbf{1}  && \mathbf{3} & & \mathbf{5}  &   \\
                    & & 0 & & 1 & & 4 & &  \\
          &0&   &0&   &3&    &&  \\
        \mathbf{0} & & \mathbf{0} & & \mathbf{3} & &  & & \\
        &&&&&&&&\\
        &&&&&&&&\\
        &&&&p&&&&
    \end{array}
\end{equation*}
\end{minipage}
\end{center}
\caption{The skew Young tableau $T$ and its corresponding Gelfand-Tsetlin pattern $p$.}
\label{fig:GTP}
\end{figure}
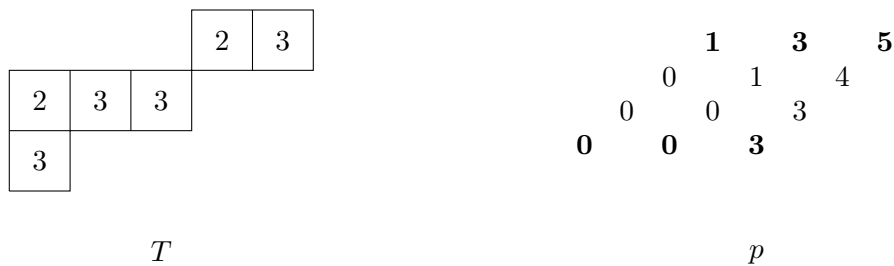

\end{example}

\begin{rem}\label{rem}
To verify the counterexample given in Example~\ref{example} also \cite[Proposition 6]{AlexanderssonIDP} can be used.
\end{rem}

\textbf{Acknowledgements:} The author would like to thank Per Alexandersson and Raman Sanyal for inspiring and fruitful discussions and many helpful comments. The author was partially supported by a Hilda Geiringer Scholarship at
the Berlin Mathematical School, the Knut and Alice Wallenberg Foundation and a Microsoft Research Fellowship of the Simons Institute for the Theory of Computing.

\bibliographystyle{siam}
\bibliography{Paper}

\end{document}